\documentclass[11pt,reqno]{amsart}
\usepackage{amsmath}
\usepackage{amsthm}
\usepackage{amssymb}
\usepackage{enumerate}
\usepackage{amscd}
\usepackage{color}
\usepackage{pb-diagram}
\usepackage{graphicx}
\usepackage[all, cmtip]{xy}
\usepackage{soul}
\usepackage{esint}
\usepackage{hyperref}
\hypersetup{pdftex,colorlinks=true,allcolors=blue}
\usepackage{hypcap}
\usepackage[titletoc]{appendix}
\usepackage{comment}

\usepackage{amsrefs}

\theoremstyle{plain}
\newtheorem{thm}{Theorem}
\newtheorem{lemma}{Lemma}
\newtheorem{prop}[lemma]{Proposition}
\newtheorem{coro}[lemma]{Corollary}
\newtheorem{theo}[lemma]{Theorem}
\newtheorem{rema}[lemma]{Remark}
\newtheorem{defi}[lemma]{Definition}
\newtheorem{thm-Intro}{Theorem} 
\newtheorem{cor-Intro}{Corollary} 

\numberwithin{equation}{section}


\begin{document}
\title[Bergman metric on the symmetrized bidisc and its consequences]{Bergman metric on the symmetrized bidisc and its consequences}

\author{Gunhee Cho}
\address{Department of Mathematics\\
University of California, Santa Barbara, \\
552 University Rd, Isla Vista, CA 93106, USA}
\email{gunhee.cho@math.ucsb.edu}

\author{Yuan Yuan}
\address{Department of Mathematics\\
Syracuse University\\
Syracuse, NY 13244, USA}

\email{yyuan05@syr.edu}


\begin{abstract} 
On the symmetrized bidisc $G_2$ with the Bergman metric, the holomorphic sectional curvature is negatively pinched and the holomorphic bisectional curvature is not. The consequences of invariant metrics are provided.
\end{abstract}

\maketitle


\section{Introduction and results}
The complete K\"ahler manifold with negatively pinched curvature is of particular interest in complex geometry (cf. \cite{WY20}). Recently, Wu and Yau obtained many deep results on complete K\"ahler manifolds with negatively pinched holomorphic sectional curvature in \cite{WY17}. In particular, they obtained the existence of complete K\"ahler-Einstein metrics with quasi-bounded geometry. Moreover, invariant metrics are shown to be equivalent. On the other hand, if the holomorphic bisectional curvature is negatively pinched, so is the holomorphic sectional curvature, but the converse is obviously not always true. There are well-known examples as homogeneous manifolds or product manifolds with negatively pinched holomorphic sectional curvature and {not} negatively pinched holomorphic bisectional curvature. It seems that it is not known whether a non-homogeneous or non-product K\"ahler manifold exists or not with negatively pinched holomorphic sectional curvature but positive holomorphic bisectional curvature somewhere and it apparently is a natural question in K\"ahler geometry  \cite{Question}. Our main result offers one complete noncompact example.
In this paper, we study  the Bergman metrics and its geometric consequences on the symmetrized bidisc $G_2$, which is neither homogeneous nor has a product structure. We will denote the unit disk in $\mathbb{C}$ by $\mathbb{D}$ and here is our result:

{\begin{theo} \label{main}
		The holomorphic sectional curvature of the Bergman metric on $G_2=\{(z_1+z_2,z_1z_2) : z_1,z_2\in \mathbb{D} \}$ is negatively pinched and the holomorphic bisectional curvature is positive somewhere.
	\end{theo}

The original motivation of the study of $G_2$ is the robust control theory and it later has been studied intensively by the functional analysts (see for example \cite{AJYNJ01,AJYNJ04,AJLZYNJ18}).
The complex geometry of the symmetrized bidisc $G_2$ is also particularly interesting (see \cite{AJYNJ01},\cite{JP13},\cite{NN07} and \cite{PPZW05}). Note that $G_2$ serves as the first non-trivial example which is not biholomorphic to any geometric convex domains but still, the Carath\'eodory-Reiffen metric and the Kobayashi-Royden metric are the same (\cite{AJYNJ04},\cite{CC04}).

One important consequence of negatively pinched holomorphic sectional curvature  in \cite{WY17} is the equivalence of invariant metrics (see Section 3). The classical invariant metrics include the Bergman metric, the Carath\'eodory-Reiffen metric, the Kobayashi-Royden metric, and the complete K\"ahler-Einstein metric of Ricci curvature equal to $-1$. Note that
 invariant metrics on K\"ahler manifolds with the uniform squeezing property  are equivalent  (cf. \cite{KLXSSTY04, SKY09}).
In particular, the equivalence of those invariant metrics has been established for strictly pseudoconvex domains \cite{IG75}, geometric convex domains \cite{Lem81, KKLZ16, GC19}, pseudoconvex domains of finite type in $\mathbb{C}^2$ \cite{CDW89,FS96,MJ89}, and $\mathbb{C}$-convex domains \cite{NNAL17} (also some pseudoconvex domains, see \cite{GC20,GK21}). Equivalence of classical invariant metrics on $ G_2$ also follows from \cite{NNAL17, NNPPZW08, SKY09}.

\section{Curvature tensors of the Bergman metric}

$G_2=\{(z_1+z_2,z_1z_2) : z_1,z_2\in \mathbb{D} \}$ is defined as the image of the bidisc $\mathbb{D}^2$ under $\Phi$, where
\begin{equation*}
\Phi : \mathbb{D}\times \mathbb{D} \rightarrow G_2, (z_1,z_2) \mapsto (z_1+z_2,z_1z_2)=:(w_1,w_2).
\end{equation*}
The Bergman kernel $B_{G_2}(w, w)$ of $G_2$ was explicit (cf. \cite{EAZW05}, \cite{MGSRSZG13}) and here we describe it by using $B=\Phi^*B_{G_2}$, the pull-back of the Bergman kernel on $\mathbb{D}^2$, given by

\begin{equation}\label{eq:pullback-bergman-kernel}
B(z,z)=\frac{1}{2\pi^2}\frac{1}{(z_1-z_2)(\overline{z_1-z_2})}\left\{\frac{1}{(1-z_1\overline{z_1})^2(1-z_2\overline{z_2})^2}-\frac{1}{(1-z_1\overline{z_2})^2(1-z_2\overline{z_1})^2}  \right\}
\end{equation}
(cf. page 12 in \cite{LCSGKYY19}).

Now we recall the characterization of the automorphism group of  $G_2$ (cf. \cite{JP04}). 
\begin{prop}
Any automorphism $H$ of $G_2$ is in the form of 
$$H(\Phi(z_1, z_2)) = \Phi(h(z_1), h(z_2))$$
for  $h \in \text{Aut}(\mathbb{D})$,
where $z_1, z_2 \in \mathbb{D}$.
\end{prop}
\begin{coro}\label{cor:invariance}
For any $(w_1, w_2)  \in G_2$, there exists $H \in Aut(G_2)$ such that $H(w_1, w_2)=(x, 0)$ 
for $x \in [0, 1)$.
\end{coro}
\begin{proof}
For any $z_1  \in \mathbb{D}$, there exists $h \in \text{Aut}(\mathbb{D})$ such that $h(z_1)=0$. For any $z_2 \in \mathbb{D}$, there exists $\theta\in [0, 2\pi)$ such that $e^{i\theta}h(z_2) = x \in [0, 1)$. 
Therefore, $ \Phi(e^{i\theta} h(z_1), e^{i\theta} h(z_2)) = (x, 0)$. This finishes the proof.
\end{proof}
Since the Bergman metric is invariant under automorphism, in order to estimate the Bergman metric and its covariant derivatives, it suffices to evaluate at $(x, 0) \in G_2$ or equivalently $(x, 0) \in  \mathbb{D} \times  \mathbb{D}$ for $x \in [0, 1)$.
We will use the coordinate $w_1=z_1+z_2,w_2=z_1 z_2$ on $G_2$ for vector fields $\frac{\partial}{\partial w_i},i=1,2$. Then the metric component of the pullback Bergman metric is given by 
\begin{equation}\label{eq:metric}
g_{i\overline{j}}=\frac{{\partial}^2 \log B_{G_2}(w,\overline{w}) }{\partial w_{i}\partial \overline{w_{j}}}={B^{-2}_{G_2}}(B_{G_2}{{\partial^2_{i\overline{j}}} B_{G_2}}-{{\partial_i}  B_{G_2}}{\partial _{\overline{j}}}{B_{G_2}}), i=1,2.
\end{equation}
We use the notation $\frac{\partial}{\partial w_{1}}=\partial_{1}$, $\frac{\partial}{\partial {\overline{w_1}}}=\partial_{\overline{1}}$, $\frac{\partial}{\partial w_2}=\partial_{2}$, $\frac{\partial}{\partial {\overline{w_2}}}=\partial_{\overline{2}}$.
To use the map $\Phi$ in computations, we convert from $\frac{\partial}{\partial z_i}$ to $\frac{\partial}{\partial w_j}$ by the inverse function theorem, and expressions of $\frac{\partial z_i}{\partial w_j}$ are given by 
\begin{align}\label{eq:conversion_matrix}
\frac{\partial z_1}{\partial w_1}&=\frac{z_1}{z_1-z_2}, 
\frac{\partial z_1}{\partial w_2}=\frac{-1}{z_1-z_2},
\frac{\partial z_2}{\partial w_1}=\frac{-z_2}{z_1-z_2},
\frac{\partial z_2}{\partial w_2}=\frac{1}{z_1-z_2},
\end{align}
where $z_1,z_2$ satisfy $w_1=z_1+z_2,w_2=z_1z_2$. Since we will use $d\Phi^{-1}=\left(\frac{\partial z_i}{\partial w_j}\right)_{i,j=1,2}$ for computations, we shall use the notation $\Phi^{-1}$ which makes sense only in the relation $B_{G_2}=B \circ  \Phi ^{-1}$ on that given point.

The following proposition follows from direct computations. 
\begin{prop}\label{derivatives_of_Bergman_kernel}
The derivatives of $B$ in \eqref{eq:pullback-bergman-kernel} at $(x,0) \in \mathbb{D}\times \mathbb{D}, 0\leq x<1$ (in a $(z_1,z_2)-coordinate$) are given by
\begin{align*}
\partial_{z_1} B&=\partial_{\overline{z}_1} B=\frac{x \left(x^2-3\right)}{2 \pi ^2 \left(x^2-1\right)^3}, \partial_{z_2} B=\partial_{\overline{z}_2} B=-\frac{x \left(2 x^2-3\right)}{2 \pi ^2 \left(x^2-1\right)^2},\\
\partial^2_{z_1\overline{z}_1}B&=\frac{-x^4+4 x^2+3}{2 \pi ^2 \left(x^2-1\right)^4}, \partial^2_{1z_1\overline{z}_2}B=\partial^2_{z_2\overline{z}_1}B=\frac{x^2-3}{2 \pi ^2 \left(x^2-1\right)^3},
\partial^2_{z_2\overline{z}_2}B=\frac{-4 x^4+4 x^2+3}{2 \pi ^2 \left(x^2-1\right)^2}, \\
\partial^2_{z_1{z_1}}B&=-\frac{x^2 \left(x^2-4\right)}{\pi ^2 \left(x^2-1\right)^4},  
\partial^2_{z_1{z_2}}B=\frac{x^2 \left(x^2-2\right)}{\pi ^2 \left(x^2-1\right)^3},
\partial^2_{z_2z_2}B=\frac{4 x^2-3 x^4}{\pi ^2 \left(x^2-1\right)^2},\\
\partial^3_{z_1\overline{z}_1z_1}B&=\frac{x \left(x^4-5 x^2-8\right)}{\pi ^2 \left(x^2-1\right)^5},
\partial^3_{z_1\overline{z}_1z_2}B=\partial^3_{z_1\overline{z}_2z_1}B=\partial^3_{z_2\overline{z}_1z_1}B=-\frac{x \left(x^2-4\right)}{\pi ^2 \left(x^2-1\right)^4},\\
\partial^3_{z_1\overline{z}_2z_2}B&=\partial^3_{z_2\overline{z}_2z_1}B=\frac{x \left(2 x^2-5\right)}{\pi ^2 \left(x^2-1\right)^3},\partial^3_{z_2\overline{z}_1z_2}B=-\frac{x \left(3 x^4-9 x^2+8\right)}{\pi ^2 \left(x^2-1\right)^3},
\partial^3_{z_2\overline{z}_2z_2}B=\frac{-6 x^5+5 x^3+4 x}{\pi ^2 \left(x^2-1\right)^2},\\
\partial^3_{z_1\overline{z}_1\overline{z}_1}B&=-\frac{-x^5+5 x^3+8 x}{\pi ^2 \left(x^2-1\right)^5},
\partial^3_{z_1\overline{z}_1\overline{z}_2}B=\partial^3_{z_1\overline{z}_2\overline{z}_1}B=\partial^3_{z_2\overline{z}_1\overline{z}_1}B={\frac{4 x-x^3}{\pi ^2 \left(x^2-1\right)^4}},\\
\partial^3_{z_1\overline{z}_2\overline{z}_2}B&=-\frac{x(3 x^4-9 x^2+8 )}{\pi ^2 \left(x^2-1\right)^3},\partial^3_{z_2\overline{z}_2\overline{z}_1}B=-\frac{5 x-2 x^3}{\pi ^2 \left(x^2-1\right)^3},\\
 \partial^3_{z_2\overline{z}_1\overline{z}_2}B&=-\frac{5 x-2 x^3}{\pi ^2 \left(x^2-1\right)^3},
\partial^3_{z_2\overline{z}_2\overline{z}_2}B=\frac{-6 x^5+5 x^3+4 x}{\pi ^2 \left(x^2-1\right)^2}.
\end{align*}
\begin{align*}
\partial^4_{z_1\overline{z}_1z_1\overline{z}_1}B&=\frac{-2 x^6+12 x^4+42 x^2+8}{\pi ^2 \left(x^2-1\right)^6}, \partial^4_{z_1\overline{z}_1z_1\overline{z}_2}B=\partial^4_{z_1\overline{z}_1z_2\overline{z}_1}B=\partial^4_{z_1\overline{z}_2z_1\overline{z}_1}B=\frac{2 \left(x^4-5 x^2-2\right)}{\pi ^2 \left(x^2-1\right)^5},  \\
\partial^4_{z_1\overline{z}_1z_2\overline{z}_2}B&=\partial^4_{z_1\overline{z}_2z_2\overline{z}_1}B=\partial^4_{z_2\overline{z}_2z_1\overline{z}_1}B=\frac{-2 x^4+6 x^2+5}{\pi ^2 \left(x^2-1\right)^4},  
\partial^4_{z_1\overline{z}_2z_1\overline{z}_1}B=-\frac{2 \left(x^2-4\right)}{\pi ^2 \left(x^2-1\right)^4},  \\
\partial^4_{z_2\overline{z}_2z_1\overline{z}_2}B&=\partial^4_{z_2\overline{z}_2z_2\overline{z}_1}B=\partial^4_{z_1\overline{z}_2z_2\overline{z}_2}B=-\frac{2 \left(3 x^6-9 x^4+7 x^2+2\right)}{\pi ^2
	\left(x^2-1\right)^3},  
\partial^4_{z_2\overline{z}_2z_2\overline{z}_2}B=\frac{2 \left(-9 x^6+6 x^4+5 x^2+4\right)}{\pi ^2
	\left(x^2-1\right)^2}.
\end{align*}
\end{prop}
\begin{rema}
	One can verify from computations that all formulas in Proposition~\ref{derivatives_of_Bergman_kernel} at $(x,0), 0\leq x<1 \in \mathbb{D}\times \mathbb{D}$ coincide at the value $(0,x), 0\leq x<1$. Hence we can use either $(x,0)$ or $(0,x)$ on $\mathbb{D}\times \mathbb{D}$ as the elements of the inverse image  of $\Phi$ at $(x,0)\in G_2$.
\end{rema}

\begin{prop}\label{prop:metric} The components of the Bergman metric $g_{i\overline{j}}$ at $(x,0), 0\leq x<1 \in G_2$ are given as follows:
\begin{align*}
g_{1\overline{1}}&=\frac{6-4 x^2}{\left(x^4-3 x^2+2\right)^2},\\
g_{1\overline{2}}&={{g_{2\overline{1}}=}}\frac{2 x \left(x^2-2\right)}{\left(x^2-1\right)^2},\\
g_{2\overline{2}}&=-\frac{2 \left(2 x^4-6 x^2+5\right)}{\left(x^2-2\right)
\left(x^2-1\right)^2}.
\end{align*}	
\begin{proof}
The first derivatives of $B\circ \Phi^{-1}$ are
\begin{equation*}
{\partial_i}  B_{G_2}=\frac{\partial}{\partial w_i}(B\circ \Phi^{-1})=\partial_{z_1} B \frac{\partial z_1}{\partial w_i}+\partial_{z_2} B \frac{\partial z_2}{\partial w_i} ,i=1,2,
\end{equation*}
and similar formulas hold for complex conjugate case. So with Proposition~\ref{derivatives_of_Bergman_kernel}, computations give that at $(x,0), 0\leq x<1$,
\begin{align*}
&{\partial_1}  B_{G_2}={\partial_{\overline{1}}}  B_{G_2}=\frac{x \left(x^2-3\right)}{2 \pi ^2 \left(x^2-1\right)^3},\\
&{\partial_2}  B_{G_2}={\partial_{\overline{2}}}  B_{G_2}=-\frac{x^2 \left(x^2-2\right)}{\pi ^2 \left(x^2-1\right)^3}.
\end{align*}
For second derivatives of $B\circ \Phi^{-1}$, since
\begin{equation*}
\frac{\partial}{\partial \overline{w}_j}\left((\partial_{z_i} B)\circ \Phi^{-1}\right)=\frac{\partial}{\partial \overline{z}_1}(\partial_{z_i} B)\frac{\partial  \overline{z}_1}{\partial \overline{w}_j}+\frac{\partial}{\partial \overline{z}_2}(\partial_{z_i} B)\frac{\partial \overline{z}_2}{\partial  \overline{w}_j},
\end{equation*}
we have 
\begin{align*}
&\partial^2_{i\overline{j}} B_{G_2}=\frac{\partial^2}{\partial w_i \partial \overline{w}_j}(B\circ \Phi^{-1})=\frac{\partial}{\partial \overline{w}_j}\left(\partial_{z_1} B \frac{\partial z_1}{\partial w_i} \right) + \frac{\partial}{\partial \overline{w}_j}\left(\partial_{z_2} B \frac{\partial z_2}{\partial w_i} \right) \\
&=\frac{\partial}{\partial \overline{w}_j}\left((\partial_{z_1} B)\circ \Phi^{-1}\right)\frac{\partial z_1}{\partial w_i}+ \frac{\partial}{\partial \overline{w}_j}\left((\partial_{z_2} B)\circ \Phi^{-1}\right)\frac{\partial z_2}{\partial w_i}+\partial_{z_1} B \frac{\partial^2 z_1}{\partial w_i \partial \overline{w}_j}+\partial_{z_2} B \frac{\partial^2 z_2}{\partial w_i \partial \overline{w}_j}\\
&=\partial^2_{z_1\overline{z}_1}B\frac{\partial \overline{z}_1}{\partial \overline{w}_j}\frac{\partial z_1}{\partial w_i}+\partial^2_{z_1\overline{z}_2}B\frac{\partial \overline{z}_2}{\partial \overline{w}_j}\frac{\partial z_1}{\partial w_i}+\partial^2_{z_2\overline{z}_1}B\frac{\partial \overline{z}_1}{\partial \overline{w}_j}\frac{\partial z_2}{\partial w_i}+\partial^2_{z_2\overline{z}_2}B\frac{\partial \overline{z}_2}{\partial \overline{w}_j}\frac{\partial z_2}{\partial w_i},
\end{align*}
because $ \frac{\partial^2 z_1}{\partial w_i \partial \overline{w}_j}= \frac{\partial^2 z_2}{\partial w_i \partial \overline{w}_j}=0$ where $i,j=1,2$.
Hence from computation with Proposition~\ref{derivatives_of_Bergman_kernel}, at $(x,0), 0\leq x<1$,
\begin{align*}
&\partial^2_{1\overline{1}} B_{G_2}=\frac{-x^4+4 x^2+3}{2 \pi ^2 \left(x^2-1\right)^4},\\
&\partial^2_{1\overline{2}} B_{G_2} {{=\partial^2_{2\overline{1}} B_{G_2}}}=\frac{x \left(x^2-4\right)}{\pi ^2 \left(x^2-1\right)^4},\\ 
&\partial^2_{2\overline{2}} B_{G_2}=\frac{-2 x^6+6 x^4-6 x^2+5}{\pi ^2 \left(x^2-1\right)^4}.
\end{align*}
Now proposition follows from computations with ~\eqref{eq:metric}.
\end{proof}
\end{prop}

\begin{rema}
	The Bergman metric was also calculated in \cite{MR3312605} in the $(z_1,z_2)$ coordinate with the different method.
\end{rema}

\begin{prop}\label{prop:inverse-metric} The components of inverse metric of the Bergman metric $g^{i\overline{j}}$ at $(x,0)\in G_2, 0\leq x< 1$ are given as follows:
\begin{align*}
g^{1\overline{1}}&=\frac{\left(x^2-2\right)^2 \left(2 x^4-6 x^2+5\right)}{2 \left(x^8-8 x^6+23
	x^4-30 x^2+15\right)},\\
g^{1\overline{2}}&=g^{2\overline{1}}=\frac{x \left(x^2-2\right)^4}{2 \left(x^8-8 x^6+23 x^4-30
	x^2+15\right)},\\
g^{2\overline{2}}&=\frac{2 x^4-7 x^2+6}{2 x^8-16 x^6+46 x^4-60 x^2+30}.
\end{align*}	
\begin{proof}
All formulas of $g^{i\overline{j}}_{B}$ at $(x,0), 0\leq x< 1$ follow from direct computations with Proposition~\ref{prop:metric}. For the record, the determinant of $g_{i\overline{j}}$ is precisely given by
\begin{align*}
\deg(g)=-\frac{4 \left(x^8-8 x^6+23 x^4-30 x^2+15\right)}{\left(x^2-2\right)^3
	\left(x^2-1\right)^2}.
\end{align*}

\end{proof}
\end{prop}

Recall that the Christoffel symbols $\Gamma^{k}_{i j}$ of a K\"ahler metric $g=(g_{i\overline{j}})$ is written in local coordinates by
\begin{equation}\label{eq:Christoffel}
\Gamma^{k}_{i j}=g^{k\overline{l}}\partial_i g_{j\overline{l}}.
\end{equation}
On $G_2$, we have the following formulas of all $\Gamma^{k}_{i j}$: 
\begin{prop}\label{prop:christoffel} The Christoffel symbols $\Gamma^{k}_{i j}$ of the Bergman metric $g_{i\overline{j}}$ at $(x,0)\in G_2, 0\leq x< 1$ are given as follows:
	\begin{align*}
	\Gamma^{1}_{1 1}&=\frac{2 x \left(x^6-2 x^4-x^2+3\right)}{\left(x^2-2\right) \left(x^2-1\right)
		\left(x^8-8 x^6+23 x^4-30 x^2+15\right)},\\
	\Gamma^{2}_{1 1}&=\frac{6 \left(x^2-2\right)}{x^8-8 x^6+23 x^4-30 x^2+15},\\
	\Gamma^{1}_{2 1}& {=\Gamma^{1}_{1 2}} =\frac{2 x^2 \left(x^2-2\right)^2}{\left(x^2-1\right) \left(x^8-8 x^6+23 x^4-30
		x^2+15\right)},\\
	\Gamma^{1}_{2 2}&=\frac{2 x^3 \left(x^2-2\right)^3}{\left(x^2-1\right) \left(x^8-8 x^6+23 x^4-30
		x^2+15\right)},\\
	\Gamma^{2}_{2 1}& {=\Gamma^{2}_{1 2}}=-\frac{x \left(x^8-10 x^6+37 x^4-62 x^2+39\right)}{\left(x^2-2\right)
		\left(x^8-8 x^6+23 x^4-30 x^2+15\right)},\\
	\Gamma^{2}_{2 2}&=\frac{2 x^2 \left(x^2-3\right) \left(x^2-2\right)^2}{x^8-8 x^6+23 x^4-30
		x^2+15}.							
	\end{align*}	
\begin{proof}
From \eqref{eq:metric}, 
\begin{align}\label{eq:first_derivative_berg_metric}
&\frac{\partial}{\partial w_i}g_{j\overline{l}}=\partial_i g_{j\overline{l}}=-2{B^{-3}_{G_2}}\partial_iB_{G_2} (B_{G_2}{{\partial^2_{j\overline{l}}} B_{G_2}}-{{\partial_j}  B_{G_2}}{\partial _{\overline{l}}}{B_{G_2}})\nonumber \\
&+{B^{-2}_{G_2}}\left(\partial_i B_{G_2} \partial^2_{j\overline{l}}B_{G_2}+B_{G_2}\partial^3_{j\overline{l}i}B_{G_2}-\partial^2_{ji}B_{G_2}\partial_{\overline{l}}B_{G_2}-\partial_j B_{G_2} \partial^2_{\overline{l}i}B_{G_2} \right).
\end{align}
Since the formulas of $\partial_{j} B_{G_2}$ are given in the proof of  Proposition~\ref{prop:metric}, we should compute $\partial^2_{j{l}} B_{G_2}$ and $\partial^3_{j\overline{l}i}B_{G_2}$ to get all formulas of Christoffel symbols. 
Elementary calculus computations with a chain-rule give for any indices $i,j,k$,
\begin{align*}
&\partial^2_{i{j}} B_{G_2}=\frac{\partial^2}{\partial w_i \partial {w}_j}(B\circ \Phi^{-1})\\
&=\partial^2_{z_1{z_1}}B\frac{\partial {z}_1}{\partial {w}_j}\frac{\partial z_1}{\partial w_i}+\partial^2_{z_1{z_2}}B\frac{\partial {z}_2}{\partial {w}_j}\frac{\partial z_1}{\partial w_i}+\partial^2_{z_2{z_1}}B\frac{\partial {z}_1}{\partial {w}_j}\frac{\partial z_2}{\partial w_i}+\partial^2_{z_2{z_2}}B\frac{\partial {z}_2}{\partial {w}_j}\frac{\partial z_2}{\partial w_i}+\partial_{z_1} B \frac{\partial^2 z_1}{\partial w_i \partial w_j}+\partial_{z_2} B \frac{\partial^2 z_2}{\partial w_i \partial w_j},\\
&\partial^3_{i\overline{j}k}B_{G_2}=\frac{\partial}{\partial w_k}\frac{\partial^2}{\partial w_i \partial \overline{w}_j}(B\circ \Phi^{-1})=\\
&\left((\partial^3_{z_1\overline{z}_1z_1}B)\frac{\partial z_1}{\partial w_k}+(\partial^3_{z_1\overline{z}_1 z_2}B)\frac{\partial z_2}{\partial w_k} \right)\frac{\partial \overline{z}_1}{\partial \overline{w}_j}\frac{\partial {z}_1}{\partial {w}_i}+\left((\partial^3_{z_1\overline{z}_2 z_1}B)\frac{\partial z_1}{\partial w_k}+(\partial^3_{z_1\overline{z}_2 z_2}B)\frac{\partial z_2}{\partial w_k} \right)\frac{\partial \overline{z}_2}{\partial \overline{w}_j}\frac{\partial {z}_1}{\partial {w}_i}\\
&+\left((\partial^3_{z_2\overline{z}_1 z_1}B)\frac{\partial z_1}{\partial w_k}+(\partial^3_{z_2\overline{z}_1 z_2}B)\frac{\partial z_2}{\partial w_k} \right)\frac{\partial \overline{z}_1}{\partial \overline{w}_j}\frac{\partial {z}_2}{\partial {w}_i}+\left((\partial^3_{z_2\overline{z}_2 z_1}B)\frac{\partial z_1}{\partial w_k}+(\partial^3_{z_2\overline{z}_2 z_2}B)\frac{\partial z_2}{\partial w_k} \right)\frac{\partial \overline{z}_2}{\partial \overline{w}_j}\frac{\partial {z}_2}{\partial {w}_i}\\
&+\partial^2_{z_1\overline{z}_1}B \frac{\partial \overline{z}_1}{\partial \overline{w}_j}\frac{\partial^2 z_1}{\partial w_i \partial w_k}+\partial^2_{z_1\overline{z}_2}B \frac{\partial \overline{z}_2}{\partial \overline{w}_j}\frac{\partial^2 z_1}{\partial w_i \partial w_k}+\partial^2_{z_2\overline{z}_1}B \frac{\partial \overline{z}_1}{\partial \overline{w}_j}\frac{\partial^2 z_2}{\partial w_i \partial w_k}+\partial^2_{z_2\overline{z}_2}B \frac{\partial \overline{z}_2}{\partial \overline{w}_j}\frac{\partial^2 z_2}{\partial w_i \partial w_k}.
\end{align*}
From above, it suffices to determine all formulas of $\frac{\partial^2 z_j}{\partial w_i \partial w_j}$. With \eqref{eq:conversion_matrix} at $(x,0)$, 
\begin{align*}
\frac{\partial^2 z_1}{\partial w_1 \partial w_1}&=0,\frac{\partial^2 z_1}{\partial w_1 \partial w_2}=\frac{1}{x^2}, \frac{\partial^2 z_1}{\partial w_2 \partial w_2}=\frac{-2}{x^3}, \\
\frac{\partial^2 z_2}{\partial w_1 \partial w_1}&=0,\frac{\partial^2 z_2}{\partial w_1 \partial w_2}=-\frac{1}{x^2}, \frac{\partial^2 z_2}{\partial w_2 \partial w_2}=\frac{2}{x^3}.
\end{align*}

Now each formula $\Gamma^{i}_{j k}$ follows from computations with putting all necessary terms in \eqref{eq:Christoffel}.
\end{proof}

\end{prop}

\begin{prop}\label{prop:curv_components} The curvature components of the Bergman metric at $(x,0)\in G_2, 0\leq x< 1$ are given by
	\begin{align*}
R_{1\overline{1}1\overline{1}}&=\frac{4 \left(9 x^{16}-108 x^{14}+551 x^{12}-1552 x^{10}+2605 x^8-2598
	x^6+1410 x^4-300 x^2-18\right)}{\left(x^4-3 x^2+2\right)^4 \left(x^8-8 x^6+23
	x^4-30 x^2+15\right)} ,\\
R_{2\overline{2}1\overline{1}}&=R_{2\overline{1}1\overline{2}}=R_{1\overline{2}2\overline{1}}=R_{1\overline{1}2\overline{2}}\\
=&\frac{4 \left(x^{16}-12 x^{14}+68 x^{12}-248 x^{10}+627 x^8-1074 x^6+1170
	x^4-726 x^2+195\right)}{\left(x^2-2\right)^3 \left(x^2-1\right)^4 \left(x^8-8
	x^6+23 x^4-30 x^2+15\right)},\\
R_{1\overline{2}1\overline{2}}&=R_{2\overline{1}2\overline{1}}\\
=&-\frac{4 x^2 \left(x^{12}-12 x^{10}+59 x^8-160 x^6+245 x^4-198
	x^2+66\right)}{\left(x^2-1\right)^4 \left(x^8-8 x^6+23 x^4-30
	x^2+15\right)},\\
R_{2\overline{1}1\overline{1}}&=R_{1\overline{2}1\overline{1}}=R_{1\overline{1}2\overline{1}}=R_{1\overline{1}1\overline{2}}\\
=&\frac{4 x \left(2 x^{10}-19 x^8+76 x^6-147 x^4+138
	x^2-51\right)}{\left(x^2-2\right) \left(x^2-1\right)^4 \left(x^8-8 x^6+23 x^4-30
	x^2+15\right)},\\
R_{1\overline{2}2\overline{2}}&=R_{2\overline{1}2\overline{2}}=R_{2\overline{2}1\overline{2}}=R_{2\overline{2}2\overline{1}}\\
=&\frac{4 x \left(x^{12}-10 x^{10}+47 x^8-130 x^6+207 x^4-174
	x^2+60\right)}{\left(x^2-1\right)^4 \left(x^8-8 x^6+23 x^4-30
	x^2+15\right)},\\
R_{2\overline{2}2\overline{2}}&=\frac{4 \left(7 x^{16}-84 x^{14}+423 x^{12}-1156 x^{10}+1829 x^8-1614 x^6+624
	x^4+60 x^2-90\right)}{\left(x^2-2\right)^2 \left(x^2-1\right)^4 \left(x^8-8 x^6+23
	x^4-30 x^2+15\right)}.
\end{align*}	
\begin{proof}
We will compute the components of curvature tensor $R=R_{a\overline{b}c\overline{d}}dz^{a}\otimes d\overline{z}^{b}\otimes dz^{c}\otimes d\overline{z}^{d}$ associated with given Hermitian metric $g$ by well-known formula:
\begin{equation}\label{eq:curv}
R_{a\overline{b}c\overline{d}}=-\frac{\partial^2 g_{a\overline{b}}}{\partial z_c \partial \overline{z}_d}+\sum_{p,q=1}^{l}g^{q\overline{p}}\frac{\partial g_{a\overline{p}}}{\partial z_c}\frac{\partial g_{q\overline{b}}}{\partial \overline{z}_d}.
\end{equation}
For the Bergman metric $g_{i\overline{j}}$ on $G_2$, we already obtained $\frac{\partial}{\partial w_i}g_{j\overline{l}}=\partial_i g_{j\overline{l}}$ in \eqref{eq:first_derivative_berg_metric}. Also, the inverse metrix was obtained in Proposition~\ref{prop:inverse-metric}. From \eqref{eq:first_derivative_berg_metric}, $\frac{\partial^2 g_{a\overline{b}}}{\partial z_c \partial \overline{z}_d}$ is written in terms of the Bergman kernel $B_{G_2}$ as follows:
\begin{align*}
\partial^2_{i\overline{j}}g_{k\overline{l}}&=6{B^{-4}_{G_2}}\partial_{\overline{j}}{B_{G_2}}\partial_{i}{B_{G_2}}{B_{G_2}}\partial^2_{k\overline{l}}{B_{G_2}}-2{B^{-3}_{G_2}}\partial^2_{i\overline{j}}{B_{G_2}}{B_{G_2}}\partial^2_{k\overline{l}}{B_{G_2}}-{{4}}{B^{-3}_{G_2}}\partial_{i}{B_{G_2}}\partial_{\overline{j}}{B_{G_2}} \partial^2_{k\overline{l}}{B_{G_2}}\\
&-2{B^{-3}_{G_2}}\partial_i {B_{G_2}} {B_{G_2}} \partial^3_{k\overline{l}\overline{j}}{B_{G_2}}-6{B^{-4}_{G_2}}\partial_{\overline{j}}{B_{G_2}}\partial_{i}{B_{G_2}}\partial_{k}{B_{G_2}}\partial_{\overline{l}}{B_{G_2}}+2{B^{-3}_{G_2}}\partial^2_{i\overline{j}}{B_{G_2}}\partial_{k}{B_{G_2}} \partial_{\overline{l}}{B_{G_2}}\\
&+2{B^{-3}_{G_2}}\partial_{i}{B_{G_2}} \partial^2_{k\overline{j}}{B_{G_2}} \partial_{\overline{l}}{B_{G_2}}+2{B^{-3}_{G_2}}\partial_{i}{B_{G_2}} \partial_{k}{B_{G_2}}\partial^2_{\overline{l}\overline{j}}{B_{G_2}}+{B^{-2}_{G_2}}\partial^2_{i\overline{j}}{B_{G_2}}\partial^2_{k\overline{l}}{B_{G_2}}\\
&+{B^{-2}_{G_2}}\partial_{i}{B_{G_2}} \partial^3_{k\overline{l}\overline{j}}{B_{G_2}}-{{B^{-2}_{G_2}}\partial_{\overline{j}}{B_{G_2}}\partial^3_{k\overline{l}i}{B_{G_2}}}+{B^{-1}_{G_2}}\partial^4_{i\overline{j}k\overline{l}}B_{G_2}\\
&+2{B^{-3}_{G_2}}\partial_{\overline{j}}{B_{G_2}}\partial^2_{ki}{B_{G_2}}\partial_{\overline{l}}{B_{G_2}}-{B^{-2}_{G_2}}\partial^3_{ki\overline{j}}{B_{G_2}}\partial_{\overline{l}}{B_{G_2}}-{B^{-2}_{G_2}}\partial^2_{ki}{B_{G_2}}\partial^2_{\overline{l}\overline{j}}{B_{G_2}}\\
&+2{B^{-3}_{G_2}}\partial_{\overline{j}}{B_{G_2}}\partial_{k}{B_{G_2}}\partial^2_{\overline{l}i}{B_{G_2}}-{B^{-2}_{G_2}}\partial^2_{k\overline{j}}{B_{G_2}}\partial^2_{\overline{l}i}{B_{G_2}}-{B^{-2}_{G_2}}\partial_k{B_{G_2}}\partial^3_{\overline{l}i\overline{j}}{B_{G_2}}.
\end{align*}
With all formulas in the proof of Proposition~\ref{prop:christoffel}, the only missing term is $\partial^4_{k\overline{l}i\overline{j}}{B_{G_2}}$, which is written as 
\begin{align*}
&\partial^4_{i\overline{j}k\overline{l}}{B_{G_2}}=\frac{\partial}{\partial \overline{w}_l}\frac{\partial^3}{\partial w_i \partial \overline{w}_j \partial w_k}(B\circ \Phi^{-1})=\\
&\left((\partial^4_{z_1\overline{z}_1 z_1\overline{z}_1}B)\frac{\partial z_1}{\partial w_k}\frac{\partial \overline{z}_1}{\partial \overline{w}_l}+(\partial^4_{z_1\overline{z}_1 z_1\overline{z}_2}B)\frac{\partial z_1}{\partial w_k}\frac{\partial \overline{z}_2}{\partial \overline{w}_l}+(\partial^4_{z_1\overline{z}_1 z_2\overline{z}_1}B)\frac{\partial z_2}{\partial w_k}\frac{\partial \overline{z}_1}{\partial \overline{w}_l}+(\partial^4_{z_1\overline{z}_1 z_2\overline{z}_2}B)\frac{\partial z_2}{\partial w_k}\frac{\partial \overline{z}_2}{\partial \overline{w}_l} \right)\frac{\partial \overline{z}_1}{\partial \overline{w}_j}\frac{\partial {z}_1}{\partial {w}_i}\\
&+\left((\partial^4_{z_1\overline{z}_2 z_1\overline{z}_1}B)\frac{\partial z_1}{\partial w_k}\frac{\partial \overline{z}_1}{\partial \overline{w}_l}+(\partial^4_{z_1\overline{z}_2 z_1\overline{z}_2}B)\frac{\partial z_1}{\partial w_k}\frac{\partial \overline{z}_2}{\partial \overline{w}_l}+(\partial^4_{z_1\overline{z}_2 z_2\overline{z}_1}B)\frac{\partial z_2}{\partial w_k}\frac{\partial \overline{z}_1}{\partial \overline{w}_l}+(\partial^4_{z_1\overline{z}_2 z_2\overline{z}_2}B)\frac{\partial z_2}{\partial w_k}\frac{\partial \overline{z}_2}{\partial \overline{w}_l} \right)\frac{\partial \overline{z}_2}{\partial \overline{w}_j}\frac{\partial {z}_1}{\partial {w}_i}\\
&+\left((\partial^4_{z_2\overline{z}_1 z_1\overline{z}_1}B)\frac{\partial z_1}{\partial w_k}\frac{\partial \overline{z}_1}{\partial \overline{w}_l}+(\partial^4_{z_2\overline{z}_1 z_1\overline{z}_2}B)\frac{\partial z_1}{\partial w_k}\frac{\partial \overline{z}_2}{\partial \overline{w}_l}+(\partial^4_{z_2\overline{z}_1 z_2\overline{z}_1}B)\frac{\partial z_2}{\partial w_k}\frac{\partial \overline{z}_1}{\partial \overline{w}_l}+(\partial^4_{z_2\overline{z}_1 z_2\overline{z}_2}B)\frac{\partial z_2}{\partial w_k}\frac{\partial \overline{z}_2}{\partial \overline{w}_l} \right)\frac{\partial \overline{z}_1}{\partial \overline{w}_j}\frac{\partial {z}_2}{\partial {w}_i}\\
&+\left((\partial^4_{z_2\overline{z}_2 z_1\overline{z}_1}B)\frac{\partial z_1}{\partial w_k}\frac{\partial \overline{z}_1}{\partial \overline{w}_l}+(\partial^4_{z_2\overline{z}_2 z_1\overline{z}_2}B)\frac{\partial z_1}{\partial w_k}\frac{\partial \overline{z}_2}{\partial \overline{w}_l}+(\partial^4_{z_2\overline{z}_2 z_2\overline{z}_1}B)\frac{\partial z_2}{\partial w_k}\frac{\partial \overline{z}_1}{\partial \overline{w}_l}+(\partial^4_{z_2\overline{z}_2 z_2\overline{z}_2}B)\frac{\partial z_2}{\partial w_k}\frac{\partial \overline{z}_2}{\partial \overline{w}_l} \right)\frac{\partial \overline{z}_2}{\partial \overline{w}_j}\frac{\partial {z}_2}{\partial {w}_i}\\
&+\left((\partial^3_{z_1\overline{z}_1 z_1}B)\frac{\partial z_1}{\partial w_k}+(\partial^3_{z_1\overline{z}_1 z_2}B)\frac{\partial z_2}{\partial w_k} \right)\frac{\partial^2 \overline{z}_1}{\partial \overline{w}_j\partial \overline{w}_l}\frac{\partial {z}_1}{\partial {w}_i}+\left((\partial^3_{z_1\overline{z}_2 z_1}B)\frac{\partial z_1}{\partial w_k}+(\partial^3_{z_1\overline{z}_2 z_2}B)\frac{\partial z_2}{\partial w_k} \right)\frac{\partial^2 \overline{z}_2}{\partial \overline{w}_j\partial \overline{w}_l}\frac{\partial {z}_1}{\partial {w}_i}\\
&+\left((\partial^3_{z_2\overline{z}_1 z_1}B)\frac{\partial z_1}{\partial w_k}+(\partial^3_{z_2\overline{z}_1 z_2}B)\frac{\partial z_2}{\partial w_k} \right)\frac{\partial^2 \overline{z}_1}{\partial \overline{w}_j\partial \overline{w}_l}\frac{\partial {z}_2}{\partial {w}_i}+\left((\partial^3_{z_2\overline{z}_2 z_1}B)\frac{\partial z_1}{\partial w_k}+(\partial^3_{z_2\overline{z}_2 z_2}B)\frac{\partial z_2}{\partial w_k} \right)\frac{\partial^2 \overline{z}_2}{\partial \overline{w}_j\partial \overline{w}_l}\frac{\partial {z}_2}{\partial {w}_i}\\
&+\partial^3_{z_1\overline{z}_1\overline{z}_1}B\frac{\partial \overline{z}_1}{\partial \overline{w}_l} \frac{\partial \overline{z}_1}{\partial \overline{w}_j}\frac{\partial^2 z_1}{\partial w_i \partial w_k}+\partial^3_{z_1\overline{z}_1\overline{z}_2}B\frac{\partial \overline{z}_2}{\partial \overline{w}_l} \frac{\partial \overline{z}_1}{\partial \overline{w}_j}\frac{\partial^2 z_1}{\partial w_i \partial w_k}+\partial^2_{z_1\overline{z}_1}B \frac{\partial^2 \overline{z}_1}{\partial \overline{w}_j \partial \overline{w}_l}\frac{\partial^2 z_1}{\partial w_i \partial w_k}\\
&+\partial^3_{z_1\overline{z}_2\overline{z}_1}B\frac{\partial \overline{z}_1}{\partial \overline{w}_l} \frac{\partial \overline{z}_2}{\partial \overline{w}_j}\frac{\partial^2 z_1}{\partial w_i \partial w_k}+\partial^3_{z_1\overline{z}_2\overline{z}_2}B\frac{\partial \overline{z}_2}{\partial \overline{w}_l} \frac{\partial \overline{z}_2}{\partial \overline{w}_j}\frac{\partial^2 z_1}{\partial w_i \partial w_k}+\partial^2_{z_1\overline{z}_2}B \frac{\partial^2 \overline{z}_2}{\partial \overline{w}_j \partial \overline{w}_l}\frac{\partial^2 z_1}{\partial w_i \partial w_k}\\
&+\partial^3_{z_2\overline{z}_1\overline{z}_1}B\frac{\partial \overline{z}_1}{\partial \overline{w}_l} \frac{\partial \overline{z}_1}{\partial \overline{w}_j}\frac{\partial^2 z_2}{\partial w_i \partial w_k}+\partial^3_{z_2\overline{z}_1\overline{z}_2}B\frac{\partial \overline{z}_2}{\partial \overline{w}_l} \frac{\partial \overline{z}_1}{\partial \overline{w}_j}\frac{\partial^2 z_2}{\partial w_i \partial w_k}+\partial^2_{{z_{2}}\overline{z}_1}B \frac{\partial^2 \overline{z}_1}{\partial \overline{w}_j \partial \overline{w}_l}\frac{\partial^2 z_2}{\partial w_i \partial w_k}\\
&+\partial^3_{z_2\overline{z}_2\overline{z}_1}B\frac{\partial \overline{z}_1}{\partial \overline{w}_l} \frac{\partial \overline{z}_2}{\partial \overline{w}_j}\frac{\partial^2 z_2}{\partial w_i \partial w_k}+\partial^3_{z_2\overline{z}_2\overline{z}_2}B\frac{\partial \overline{z}_2}{\partial \overline{w}_l} \frac{\partial \overline{z}_2}{\partial \overline{w}_j}\frac{\partial^2 z_2}{\partial w_i \partial w_k}+\partial^2_{z_2\overline{z}_2}B \frac{\partial^2 \overline{z}_2}{\partial \overline{w}_j \partial \overline{w}_l}\frac{\partial^2 z_2}{\partial w_i \partial w_k}.
\end{align*}
Then each formula of $R_{a\overline{b}c\overline{d}}$ can be obtained from elementary but lengthy computations. 
\end{proof}
\end{prop}

To compute the holomorphic sectional curvature of the Bergman metric on $G_2$, we proceed with the Gram-Schmidts process to determine the orthonormal basis $X,Y$. Take the first unit vector field 
\begin{equation}\label{eq:orthonormal_X}
X=\frac{\partial_1}{\sqrt{g_{1\overline{1}}}}.
\end{equation}
Then another vector field $\tilde{Y}$ which is orthogonal to $X$ is given by 
\begin{equation*}
\tilde{Y}=\frac{\partial_2}{\sqrt{g_{2\overline{2}}}}-g(\frac{\partial_2}{\sqrt{g_{2\overline{2}}}},X)X=a_1 \partial_1+a_2 \partial_2,
\end{equation*}
where $a_1=-\frac{g_{2\overline{1}}}{g_{1\overline{1}}\sqrt{g_{2\overline{2}}}}$, $a_2=\frac{1}{\sqrt{g_{2\overline{2}}}}$.
Since $g({{\tilde Y}}, {{\tilde Y}})=a_1 \overline{a_1}g_{1\overline{1}}+a_1\overline{a_2}g_{1\overline{2}}+a_2\overline{a_1}g_{2\overline{1}}+a_2\overline{a_2}g_{2\overline{2}}$, we will use
\begin{equation}\label{eq:orthonormal_Y}
Y=\frac{\tilde{Y}}{\sqrt{g(\tilde{Y},\tilde{Y})}}=\frac{a_1 \partial_1+a_2 \partial_2}{\sqrt{a_1 \overline{a_1}g_{1\overline{1}}+a_1\overline{a_2}g_{1\overline{2}}+a_2\overline{a_1}g_{2\overline{1}}+a_2\overline{a_2}g_{2\overline{2}}}}=:t_1\partial_1+t_2\partial_2,
\end{equation}
where
\begin{equation}\label{eq:t_i}
t_i=\frac{a_i}{{\sqrt{a_1 \overline{a_1}g_{1\overline{1}}+a_1\overline{a_2}g_{1\overline{2}}+a_2\overline{a_1}g_{2\overline{1}}+a_2\overline{a_2}g_{2\overline{2}}}}}, i=1,2.
\end{equation}

\begin{prop}\label{prop:hsc} Let $H(Z)=R(Z, \bar Z, Z, \bar Z)$ for $Z \in \{X, Y\}$.  
The holomorphic sectional curvatures $H(X), H(Y)$ of the Bergman metric at $(x,0)\in G_2,0\leq x <1$ are given as below:
	\begin{align*}
&	H(X)=\frac{9 x^{16}-108 x^{14}+551 x^{12}-1552 x^{10}+2605 x^8-2598 x^6+1410
	x^4-300 x^2-18}{\left(3-2 x^2\right)^2 \left(x^8-8 x^6+23 x^4-30
	x^2+15\right)},\\
&H(Y)\left(3-2 x^2\right)^2 \left(x^4-5
	x^2+5\right)^3 \left(x^4-3 x^2+3\right)^2\\
	&=9 x^{28}-225 x^{26}+2575 x^{24}-17844 x^{22}+83491 x^{20}\\
	&-278485x^{18}+681267 x^{16}-1237584 x^{14}+1668725 x^{12}-1646775 x^{10}\\
	&+1150505x^8-531240 x^6+137820 x^4-9810 x^2-2430.
	\end{align*}
	In particular, all values of $H(X)$ and $H(Y)$  are negative at $(x,0)\in G_2,0\leq x <1$ and 
	\begin{equation*}
	\lim_{x \rightarrow 1}H(X)=\lim_{x \rightarrow 1}H(Y)=-1.
	\end{equation*}
\begin{proof}
From the definition of the holomorphic sectional curvature, compute $H(X), H(Y)$ which become
		\begin{equation*}
		H(X)=\frac{R_{1\overline{1}1\overline{1}}}{g_{1\overline{1}}\overline{g_{1\overline{1}}}},
		\end{equation*}
		and
		\begin{equation*}
		H(Y)=\sum^{2}_{i,j,k,l=1} t_{i}\overline{t_{j}}t_{k}\overline{t_{l}} R_{i\overline{j}k\overline{l}}.
		\end{equation*}
Then formulas of $H(X), H(Y)$ follow from the direct elementary computations and one can check that all values of $H(X), H(Y)$ are negative. 		
	\end{proof}
\end{prop}

However, we can also compute the bisectional curvature of the Bergman metric on $G_2$ based on Proposition~\ref{prop:curv_components}. 
\begin{prop}\label{prop:hsc2}
Let $B(X,Y):={R(X, \bar X,Y, \bar Y)}$. Then at $(x,0)\in G_2, 0\leq x<1$, 
\begin{equation*}
B(X,Y)=-\frac{\left(x^2-1\right)^2 f_1(x)}{\left(3-2 x^2\right)^2 \left(x^8-8 x^6+23 x^4-30
	x^2+15\right)^2},
\end{equation*}
where 
\begin{align*}
f_1(x)&=9 x^{20}-162 x^{18}+1297 x^{16}-6074
x^{14}+18412 x^{12}-37738 x^{10}+52968 x^8\\
&-50274 x^6+30876 x^4-11070x^2+1755.
\end{align*}
In particular, 
\begin{align*}
\lim_{x \rightarrow 1}B(X,Y)&=0,\\
B(X,Y)(0.9,0.9,0,0)&=0.00679073.
\end{align*}	
Consequently, the bisectional curvature of the Bergman metric on $G_2$ is not negatively pinched.
\begin{proof}
By \eqref{eq:orthonormal_X} and \eqref{eq:orthonormal_Y},
\begin{align*}
B(X,Y)=\frac{t_1 \overline{t_1}}{g_{1\overline{1}}}R_{1\overline{1}1\overline{1}}+\frac{t_1 \overline{t_2}}{g_{1\overline{1}}}R_{1\overline{1}1\overline{2}}+\frac{t_2 \overline{t_2}}{g_{1\overline{1}}}R_{1\overline{1}2\overline{2}}+\frac{t_2 \overline{t_1}}{g_{1\overline{1}}}R_{1\overline{1}2\overline{1}}.
\end{align*}
Now proposition follows from direct computations with Proposition~\ref{prop:curv_components} and \eqref{eq:t_i}.
\end{proof}

\end{prop}

It follows by the similar argument that

\begin{lemma}\label{prop:hsc33}
At $(x,0)\in G_2, 0\leq x<1$, 
\begin{align*}
&R(X, \bar X, X, \bar Y)=R(X, \bar X, Y, \bar X)= \\
&~~~~~~-\frac{3 x \left(2-x^2\right)^\frac{5}{2} \left(1-x^2\right)^3 \left(3 x^8-24 x^6+71 x^4-92
	x^2+45\right)}{ \left(3-2 x^2\right)^2 \sqrt{(2 x^4-6
			x^2+5)(3-2 x^2)} \left(4 x^6-18 x^4+28 x^2-15\right)
(f_2(x))^{\frac{3}{2}}} ,\\
&R(Y, \bar Y, X, \bar Y)=R(Y, \bar Y, Y, \bar X)= \\
&~~~~~~{\frac{x \left(2-x^2\right)^\frac{5}{2} \left(x^2-1\right)^2 \left(9 x^{14}-126 x^{12}+739
	x^{10}-2335 x^8+4276 x^6-4545 x^4+2610 x^2-630\right)}{\left(3-2 x^2\right)^2  \sqrt{(2 x^4-6
			x^2+5)(3-2 x^2)}
	 \left(x^4-5
	x^2+5\right)^2 \left(x^4-3
	x^2+3\right) \sqrt{f_2(x)}}},\\
&R(X, \bar Y, X, \bar Y)=R(Y, \bar X, Y, \bar X) =
{-\frac{3 x^2 \left(x^2-2\right)^3 \left(x^2-1\right)^2 \left(3 x^8-27 x^6+89
	x^4-124 x^2+62\right)}{\left(3-2 x^2\right)^2 \left(x^4-5 x^2+5\right)^2
	\left(x^4-3 x^2+3\right)}},
\end{align*}
where  
$$f_2(x)=	-\frac{ x^8-8 x^6+23 x^4-30 x^2+15}{4
	x^6-18 x^4+28 x^2-15}.$$

\end{lemma}

Now we are ready to prove the main result of the paper.

\begin{proof}[Proof of Theorem~\ref{main}]
Take any unit vector field $V=aX+bY$ with respect to the Bergman metric with $|a|^2+|b|^2=1$. Then at $(x,0)\in G_2, 0\leq x<1$, 
\begin{align}\label{eq:general_expression_hsc}
&R(V, \bar V,V, \bar V)=|a|^4 R(X,\bar X,X, \bar X)+|a|^2\overline{a}b R(Y, \bar X,X, \bar X)+|a|^2a\overline{b}R(X,\bar Y,X, \bar X)\\ \nonumber
&+|a|^2|b|^2R(Y,\bar Y,X, \bar X)+|a|^2\overline{a}bR(X,\bar X,Y, \bar X)+\overline{a}^2b^2R(Y, \bar X,Y, \bar X)+|a|^2|b|^2R(X, \bar Y,Y, \bar X)\\ \nonumber
&+\overline{a}b|b|^2R(Y, \bar Y,Y, \bar X)+|a|^2a\overline{b}R(X,\bar X,X, \bar Y)+|a|^2|b|^2R(Y, \bar X,X, \bar Y)+a^2\overline{b}^2R(X, \bar Y,X, \bar Y)\\ \nonumber
&+a\overline{b}|b|^2R(Y, \bar Y,X, \bar Y)+|a|^2|b|^2R(X, \bar X,Y, \bar Y)+\overline{a}b|b|^2R(Y, \bar X,Y, \bar Y)+a\overline{b}|b|^2R(X, \bar Y,Y, \bar Y)\\ \nonumber
&+|b|^4R(Y, \bar Y,Y, \bar Y) \\ \nonumber
&=|a|^4 H(X) + |b|^4 H(Y) + 4 |a|^2|b|^2 B(X,Y) + 4 \text{Re}(\overline{a}b) \left( |a|^2 R(X, \bar X, X, \bar Y)+ |b|^2 R(Y, \bar Y, Y, \bar X)\right) \\ \nonumber
&  + 2 \text{Re}(\overline{a}^2b^2)R(Y, \bar X,Y, \bar X). \nonumber
\end{align}

With Proposition \ref{prop:hsc}, Proposition \ref{prop:hsc2}, and Lemma  \ref{prop:hsc33}, one can show that $R(V, \bar V,V, \bar V)$ is negatively pinched for $x \in [0, 1)$. 
In fact, letting $L(V) = \left(3-2 x^2\right)^2 R(V, \bar V,V, \bar V) $, one can show that  $-10\leq L(V) \leq -1/2$. By Corollary \ref{cor:invariance}, the holomorphic sectional curvature of the Bergman metric on $G_2$ is negatively pinched between $-10$ and $-1/18$.
Lastly, the bisectional curvature condition follows from Proposition~\ref{prop:hsc2}.	
\end{proof}

\begin{rema}
It is obvious that the Bergman metric on $G_2$ is not K\"ahler-Einstein and thus $G_2$ is not a homogeneous domain.
\end{rema}

{
However, we cannot obtain a compact example by taking the quotient of $G_2$ as $G_2$ does not even admit a quotient with finite volume. One may apply \cite{Fr} to conclude $G_2$ does not admit a compact quotient.
Here in order to apply Theorem 1.6 in \cite{LW}, it suffices to verify the following simple fact.
\begin{prop}
$G_2$ is contractible.
\begin{proof}
It suffices to show that the identity map is homotopic to the constant map sending $G_2$ to $0 \in G_2$.
Let $F: [0, 1] \times G_2 \rightarrow G_2$ given by $F(t, w_1, w_2)=(tw_1, t^2w_2)$. Suppose $(w_1, w_2)= \Phi (z_1, z_2) = (z_1+z_2, z_1z_2)$ for $(z_1, z_2) \in \Delta^2$. Then $\Phi(tz_1, tz_2)= (tz_1+tz_2, t^2z_1 z_2)=(tw_1, t^2w_2)$. It follows that $F$ is a well-defined continuous map and thus the identity map and the constant map are homotopic. 
\end{proof}
\end{prop}
}

\section{Complex geometric consequences}
We study the complete K\"ahler-Einstein metric as well as other invariant metrics on $G_2$ and we have the following corollaries by applying the fundamental results proved in \cite{WY17}:

\begin{coro}\label{thm:equivalence}
The Bergman metric $g^B_{G_2}$, the Kobayashi-Royden metric $g^K_{G_2}$ and the complete K\"ahler-Einstein metric $g^{KE}_{G_2}$ 
with Ricci curvature equal to $-1$ on the symmetrized bidisc $G_2$ are uniformly equivalent.
\end{coro}

\begin{proof}[Proof of Corollary~\ref{thm:equivalence}]
With Theorem~\ref{main},  Corollary~\ref{thm:equivalence} follows from Theorem 2 and Theorem 3 in \cite{WY17}. 
\end{proof}

\begin{rema}
We are kindly informed by Nikolai Nikolov that this result is known by the property of the squeezing functions on $\mathbb{C}$-convex domains (cf. \cite{NNAL17,NNPPZW08}). 
\end{rema}

The next corollary is motivated by Example 5.1 and 5.2 in \cite{WY18} and the proof also follows from the argument there. 

\begin{coro}\label{cor:new_class}
Given any complete K\"ahler manifold $(X, g_X)$ such that the holomorphic sectional curvature is between two negative numbers,
the holomorphic sectional curvature of the product metric $g^B_{G_2} \oplus g_X$ on $\Omega:=G_2\times X$ is between two negative numbers. 
As a consequence, any closed complex submanifold $S$ of $\Omega$ admits the unique complete K\"ahler-Einstein metric $g_S^{KE}$ 
with Ricci curvature equal to $-1$. Moreover,  $g_S^{KE}$, the Kobayashi-Royden metric $g_S^K$ are uniformly equivalent.
\end{coro} 

\begin{proof}
It follows from Theorem~\ref{main} that the holomorphic sectional curvature of $g^B_{G_2} \oplus g_X$ is negatively pinched. By Lemma 13 in \cite{WY17}, there exists a complete K\"ahler metric $g_{\Omega}$ on $\Omega$ such that the holomorphic sectional curvature of  $g_{\Omega}$ is negatively pinched and $g_{\Omega}$ has the quasi-bounded geometry. Therefore, the second fundamental form of $S$ with respect to the restriction $g_\Omega|_S$ is bounded. By the decreasing property for holomorphic sectional curvature and the Gauss-Codazzi equation,  the holomorphic sectional curvature of $g_\Omega|_S$ is negatively pinched. The conclusion follows from Theorem 2 and Theorem 3 in \cite{WY17}.
\end{proof}

{\bf Acknowledgement:} The second author is supported by National Science Foundation grant DMS-1412384, Simons Foundation grant \#429722 and CUSE grant program at Syracuse University. Both authors thank Damin Wu for the very helpful discussions and thank Lixin Shen for his help on the numerical analysis. We also thank Nikolai Nikolov for informing us his results \cite{NNAL17, NNPPZW08} on the invariant metrics on the symmetrized bidisc.

\bibliographystyle{spmpsci}
\bibliography{reference}

\begin{bibdiv}
\begin{biblist}

\bib{Question}{misc}{
  title={https://mathoverflow.net/questions/288257/relationship-between-the-signs-of-different-notions-of-curvature-in-complex-geoms},
}

\bib{AJYNJ01}{article}{
      author={Agler, J.},
      author={Young, N.~J.},
       title={A {S}chwarz lemma for the symmetrized bidisc},
        date={2001},
        ISSN={0024-6093},
     journal={Bull. London Math. Soc.},
      volume={33},
      number={2},
       pages={175\ndash 186},
         url={https://doi.org/10.1112/blms/33.2.175},
      review={\MR{1815421}},
}

\bib{AJYNJ04}{article}{
      author={Agler, J.},
      author={Young, N.~J.},
       title={The hyperbolic geometry of the symmetrized bidisc},
        date={2004},
        ISSN={1050-6926},
     journal={J. Geom. Anal.},
      volume={14},
      number={3},
       pages={375\ndash 403},
         url={https://doi.org/10.1007/BF02922097},
      review={\MR{2077158}},
}

\bib{AJLZYNJ18}{article}{
      author={Agler, Jim},
      author={Lykova, Zinaida~A.},
      author={Young, N.~J.},
       title={Algebraic and geometric aspects of rational {$\Gamma$}-inner
  functions},
        date={2018},
        ISSN={0001-8708},
     journal={Adv. Math.},
      volume={328},
       pages={133\ndash 159},
         url={https://doi.org/10.1016/j.aim.2017.12.018},
      review={\MR{3771126}},
}

\bib{CDW89}{article}{
      author={Catlin, David~W.},
       title={Estimates of invariant metrics on pseudoconvex domains of
  dimension two},
        date={1989},
        ISSN={0025-5874},
     journal={Math. Z.},
      volume={200},
      number={3},
       pages={429\ndash 466},
         url={https://doi.org/10.1007/BF01215657},
      review={\MR{978601}},
}

\bib{LCSGKYY19}{article}{
      author={Chen, Liwei},
      author={Krantz, Steven~G},
      author={Yuan, Yuan},
       title={Lp regularity of the bergman projection on domains covered by the
  polydisk},
        date={2020},
        ISSN={0022-1236},
     journal={J. Funct. Anal},
         url={https://doi.org/10.1016/j.jfa.2020.108522},
}

\bib{GC20}{article}{
      author={Cho, Gunhee},
       title={Quasi-bounded geometry of the bergman metric and equivalence of
  invariant metrics},
        date={2020},
     journal={arXiv:2009.13027},
}

\bib{GC19}{article}{
      author={Cho, Gunhee},
       title={Invariant metrics on the complex ellipsoid},
        date={2021},
        ISSN={1050-6926},
     journal={J. Geom. Anal.},
      volume={31},
      number={2},
       pages={2088\ndash 2104},
  url={https://doi-org.proxy.library.ucsb.edu:9443/10.1007/s12220-019-00333-w},
      review={\MR{4215285}},
}

\bib{GK21}{article}{
      author={Cho, Gunhee},
      author={Lee, Kyu-Hwan},
       title={A lower bound of the integrated carath\'eodory--reiffen metric
  and invariant metrics on complete noncompact k\"ahler manifolds},
        date={2021},
     journal={arXiv:2109.14473},
}

\bib{CC04}{article}{
      author={Costara, C.},
       title={The symmetrized bidisc and {L}empert's theorem},
        date={2004},
        ISSN={0024-6093},
     journal={Bull. London Math. Soc.},
      volume={36},
      number={5},
       pages={656\ndash 662},
         url={https://doi.org/10.1112/S0024609304003200},
      review={\MR{2070442}},
}

\bib{EAZW05}{article}{
      author={Edigarian, Armen},
      author={Zwonek, Wlodzimierz},
       title={Geometry of the symmetrized polydisc},
        date={2005},
     journal={Arch. Math. (Basel).},
      volume={84},
      number={4},
       pages={364\ndash 374},
         url={https://doi.org/10.1007/s00013-004-1183-z},
}

\bib{Fr}{article}{
      author={Frankel, Sidney},
       title={Complex geometry of convex domains that cover varieties},
        date={1989},
        ISSN={0001-5962},
     journal={Acta Math.},
      volume={163},
      number={1-2},
       pages={109\ndash 149},
         url={https://doi.org/10.1007/BF02392734},
      review={\MR{1007621}},
}

\bib{FS96}{article}{
      author={Fu, Siqi},
       title={Geometry of {R}einhardt domains of finite type in {$\bold C^2$}},
        date={1996},
        ISSN={1050-6926},
     journal={J. Geom. Anal.},
      volume={6},
      number={3},
       pages={407\ndash 431 (1997)},
         url={https://doi.org/10.1007/BF02921658},
      review={\MR{1471899}},
}

\bib{IG75}{article}{
      author={Graham, Ian},
       title={Boundary behavior of the {C}arath\'{e}odory and {K}obayashi
  metrics on strongly pseudoconvex domains in {$C^{n}$} with smooth boundary},
        date={1975},
        ISSN={0002-9947},
     journal={Trans. Amer. Math. Soc.},
      volume={207},
       pages={219\ndash 240},
         url={https://doi.org/10.2307/1997175},
      review={\MR{372252}},
}

\bib{JP04}{article}{
      author={Jarnicki, Marek},
      author={Pflug, Peter},
       title={On automorphisms of the symmetrized bidisc},
        date={2004},
     journal={Arch. Math. (Basel)},
      volume={83},
      number={3},
       pages={264\ndash 266},
}

\bib{JP13}{book}{
      author={Jarnicki, Marek},
      author={Pflug, Peter},
       title={Invariant distances and metrics in complex analysis},
     edition={extended},
      series={De Gruyter Expositions in Mathematics},
   publisher={Walter de Gruyter GmbH \& Co. KG, Berlin},
        date={2013},
      volume={9},
        ISBN={978-3-11-025043-5; 978-3-11-025386-3},
         url={https://doi.org/10.1515/9783110253863},
      review={\MR{3114789}},
}

\bib{KKLZ16}{article}{
      author={Kim, Kang-Tae},
      author={Zhang, Liyou},
       title={On the uniform squeezing property of bounded convex domains in
  {$\Bbb{C}^n$}},
        date={2016},
        ISSN={0030-8730},
     journal={Pacific J. Math.},
      volume={282},
      number={2},
       pages={341\ndash 358},
         url={https://doi.org/10.2140/pjm.2016.282.341},
      review={\MR{3478940}},
}

\bib{Lem81}{article}{
      author={Lempert, L\'{a}szl\'{o}},
       title={La m\'{e}trique de {K}obayashi et la repr\'{e}sentation des
  domaines sur la boule},
        date={1981},
        ISSN={0037-9484},
     journal={Bull. Soc. Math. France},
      volume={109},
      number={4},
       pages={427\ndash 474},
         url={http://www.numdam.org/item?id=BSMF_1981__109__427_0},
      review={\MR{660145}},
}

\bib{KLXSSTY04}{article}{
      author={Liu, Kefeng},
      author={Sun, Xiaofeng},
      author={Yau, Shing-Tung},
       title={Canonical metrics on the moduli space of {R}iemann surfaces.
  {I}},
        date={2004},
        ISSN={0022-040X},
     journal={J. Differential Geom.},
      volume={68},
      number={3},
       pages={571\ndash 637},
         url={http://projecteuclid.org/euclid.jdg/1116508767},
      review={\MR{2144543}},
}

\bib{LW}{misc}{
      author={Liu, Kefeng},
      author={Wu, Yunhui},
       title={Geometry of complex bounded domains with finite-volume
  quotients},
        date={arXiv:1801.00459 [math.DG], 2018},
}

\bib{MJ89}{article}{
      author={McNeal, Jeffery~D.},
       title={Holomorphic sectional curvature of some pseudoconvex domains},
        date={1989},
        ISSN={0002-9939},
     journal={Proc. Amer. Math. Soc.},
      volume={107},
      number={1},
       pages={113\ndash 117},
         url={https://doi.org/10.2307/2048043},
      review={\MR{979051}},
}

\bib{MGSRSZG13}{article}{
      author={Misra, Gadadhar},
      author={Syham~Roy, Subrata},
      author={Zhang, Genkai},
       title={Reproducing kernel for a class of weighted {B}ergman spaces on
  the symmetrized polydisc},
        date={2013},
        ISSN={0002-9939},
     journal={Proc. Amer. Math. Soc.},
      volume={141},
      number={7},
       pages={2361\ndash 2370},
         url={https://doi.org/10.1090/S002-9939-2013-11514-5},
      review={\MR{3043017}},
}

\bib{NNAL17}{article}{
      author={Nikolov, N.},
      author={Andreev, L.},
       title={Boundary behavior of the squeezing functions of
  {$\Bbb{C}$}-convex domains and plane domains},
        date={2017},
        ISSN={0129-167X},
     journal={Internat. J. Math.},
      volume={28},
      number={5},
       pages={1750031, 5},
         url={https://doi.org/10.1142/S0129167X17500318},
      review={\MR{3655077}},
}

\bib{NN07}{article}{
      author={Nikolov, Nikolai},
      author={Pflug, Peter},
      author={Zwonek, Wlodzimierz},
       title={The {L}empert function of the symmetrized polydisc in higher
  dimensions is not a distance},
        date={2007},
        ISSN={0002-9939},
     journal={Proc. Amer. Math. Soc.},
      volume={135},
      number={9},
       pages={2921\ndash 2928},
         url={https://doi.org/10.1090/S0002-9939-07-08817-X},
      review={\MR{2317970}},
}

\bib{NNPPZW08}{article}{
      author={Nikolov, Nikolai},
      author={Pflug, Peter},
      author={Zwonek, Wlodzimierz},
       title={An example of a bounded {${\bf C}$}-convex domain which is not
  biholomorphic to a convex domain},
        date={2008},
        ISSN={0025-5521},
     journal={Math. Scand.},
      volume={102},
      number={1},
       pages={149\ndash 155},
         url={https://doi.org/10.7146/math.scand.a-15056},
      review={\MR{2420684}},
}

\bib{PPZW05}{article}{
      author={Pflug, Peter},
      author={Zwonek, Wlodzimierz},
       title={Description of all complex geodesics in the symmetrized bidisc},
        date={2005},
        ISSN={0024-6093},
     journal={Bull. London Math. Soc.},
      volume={37},
      number={4},
       pages={575\ndash 584},
         url={https://doi.org/10.1112/S0024609305004418},
      review={\MR{2143737}},
}

\bib{MR3312605}{article}{
      author={Trybu\l~a, Maria},
       title={Invariant metrics on the symmetrized bidisc},
        date={2015},
        ISSN={1747-6933},
     journal={Complex Var. Elliptic Equ.},
      volume={60},
      number={4},
       pages={559\ndash 565},
  url={https://doi-org.proxy.library.ucsb.edu:9443/10.1080/17476933.2014.948543},
      review={\MR{3312605}},
}

\bib{WY18}{article}{
      author={Wu, Damin},
      author={Yau, Shing-Tung},
       title={Complete {K}\"{a}hler-{E}instein metrics under certain
  holomorphic covering and examples},
        date={2018},
        ISSN={0373-0956},
     journal={Ann. Inst. Fourier (Grenoble)},
      volume={68},
      number={7},
       pages={2901\ndash 2921},
         url={http://aif.cedram.org/item?id=AIF_2018__68_7_2901_0},
      review={\MR{3959099}},
}

\bib{WY20}{article}{
      author={Wu, Damin},
      author={Yau, Shing-Tung},
       title={Some negatively curved complex geometry},
        date={2019},
     journal={Proceedings of the 8th ICCM},
}

\bib{WY17}{article}{
      author={Wu, Damin},
      author={Yau, Shing-Tung},
       title={Invariant metrics on negatively pinched complete kähler
  manifolds},
        date={2020},
        ISSN={1088-6834},
     journal={J. Amer. Math. Soc.},
      number={33},
       pages={103\ndash 133},
         url={https://doi.org/10.1090/jams/933},
}

\bib{SKY09}{article}{
      author={Yeung, Sai-Kee},
       title={Geometry of domains with the uniform squeezing property},
        date={2009},
        ISSN={0001-8708},
     journal={Adv. Math.},
      volume={221},
      number={2},
       pages={547\ndash 569},
         url={https://doi.org/10.1016/j.aim.2009.01.002},
      review={\MR{2508930}},
}

\end{biblist}
\end{bibdiv}

\end{document}